\pdfoutput=1

\documentclass{amsart}

\usepackage{verbatim}
\usepackage{xcolor}
\usepackage{tikz}
\usepackage{tikz-cd}
\usepackage{arydshln}
\usepackage{caption}
\usepackage{subcaption}
\usepackage{graphicx}
\usepackage[hyphens]{url}
\usepackage{amsmath}
\usepackage{bbm}
\usepackage{hyperref}
\usepackage{mathrsfs}

\newtheorem{theorem}{Theorem}[section]
\newtheorem{proposition}[theorem]{Proposition}

\newtheorem{corollary}[theorem]{Corollary}

\theoremstyle{definition}

\theoremstyle{remark}

\newtheorem{question}[theorem]{Question}

\numberwithin{equation}{section}

\newcommand{\dimH}{\smash{\mathrm{dim}_H}}
\newcommand{\of}{\smash{\overline{f}^\mu}}

\newcounter{count}

\newcounter{counterk} 
 
\newcounter{counterl} 
 
\newcounter{counterc} 
 
\newcounter{countercc} 
 
\newcounter{countere} 
 
\newcounter{counterd} 
 
\newcounter{countern} 
 
\newcounter{counterr}
 
\begin{document}

\title[Teichmüller horocycle flow orbit closures]{Bounds on the Hausdorff dimension of Teichmüller horocycle flow orbit closures}

\author{Francisco Arana--Herrera}

\begin{abstract}
	We show that the Hausdorff dimension of any proper Teichmüller horocycle flow orbit closure on any $\mathrm{SL}{(2,\mathbf{R})}$-invariant subvariety of Abelian or quadratic differentials is bounded away from the dimension of the subvariety in terms of the polynomial mixing rate of the Teichmüller horocycle flow on the subvariety. The proof is based on abstract methods for measurable flows adapted from work of Bourgain and Katz on sparse ergodic theorems.
\end{abstract}

\maketitle


\thispagestyle{empty}

\tableofcontents

\section{Introduction}

In \cite{CSW20}, Chaika, Smillie, and Weiss constructed the first known examples of Teichmüller horocycle flow orbit closures of non-integer Hausdorff dimension on strata of Abelian differentials. The existence of such orbit closures is in stark contrast with the rigidity of unipotent flows on homogeneous spaces, for which, by work of Ratner \cite{Rat94}, orbit closures are always embedded homogeneous submanifolds. 

A well known construction shows that diagonalizable flows on homogeneous spaces are incredibly flexible in the sense that they admit orbit closures of any Hausdorff dimension between 1 and the dimension of the underlying space. The same construction can also be applied to the Teichmüller geodesic flow on strata of Abelian or quadratic differentials.

Given the examples of Chaika, Smillie, and Weiss, it seems natural to consider the question of whether the Teichmüller horocycle flow can be as flexible as the Teichmüller geodesic flow: Does the Teichmüller horocycle flow admit orbit closures of any Hausdorff dimension between $1$ and the dimension of the ambient $\mathrm{SL}(2,\mathbf{R})$-invariant subvariety? In this paper we give a negative answer to this question by proving the following rigidity result.

\begin{theorem}
	\label{theo:main} 
	For every $\beta$-dimensional, $\mathrm{SL}(2,\mathbf{R})$-invariant subvariety of Abelian or quadratic differentials, there exists $\sigma > 0$ such that every proper Teichmüller horocycle flow orbit closure on the subvariety has Hausdorff dimension $\leq \beta - \sigma$.
\end{theorem}

Theorem \ref{theo:main} is deduced directly from the following result of independent interest.

\begin{theorem}
	\label{theo:sub_main} 
	For every $\beta$-dimensional, $\mathrm{SL}(2,\mathbf{R})$-invariant subvariety of Abelian or quadratic differentials, there exists a constant $\sigma > 0$ such that the set of points of the subvariety whose Teichmüller horocycle flow orbit does not equidistribute with respect to the affine measure of the subvariety has Hausdorff dimension $\leq \beta - \sigma$.
\end{theorem}

The gap $\sigma > 0$ in Theorems \ref{theo:main} and \ref{theo:sub_main} depends only on the polynomial mixing rate of the Teichmüller horocycle flow on the corresponding $\mathrm{SL}(2,\mathbf{R})$-invariant subvariety. By work of Ratner \cite{Ra87}, this rate of mixing is directly related to the spectral gap of the corresponding $\mathrm{SL}(2,\mathbf{R})$-representation. This gap is known to be positive for $\mathrm{SL}(2,\mathbf{R})$-invariant subvarieties of Abelian or quadratic differentials by work of Avila and Gouëzel \cite{AG13}. This work builds on previous results of Avila, Gouëzel, and Yoccoz \cite{AGY06}, and Avila and Resende \cite{AR12}.

The proof of Theorems \ref{theo:main} and \ref{theo:sub_main} is based on abstract methods for measurable flows adapted from work of Bourgain \cite{B89} and Katz \cite{K21} on sparse ergodic theorems. The main idea is that the polynomial mixing rate of a flow constraints the $L^2$-norm of its orbit averages, which in turn constraints the measure of the set of points whose orbit averages deviate from the corresponding means. Such measure bounds can be used to construct tight covers of the corresponding sets under the assumption that the flow diverges at most at a polynomial rate. 

\subsection*{Open questions.} The tension between the examples of Chaika, Smillie, and Weiss and Theorem \ref{theo:main} suggests there is much of the behavior of the Teichmüller horocycle flow we have yet to understand. Here we advertise a few related open questions.

\begin{question}
	The examples of Teichmüller horocycle flow orbit closures due to Chaika, Smillie, and Weiss are known to have Hausdorff dimension in the interval $[5.5,6)$. What is their exact Hausdorff dimension?
\end{question}

\begin{question}
	\label{question:2}
	Does there exist an Abelian or quadratic differential which does not belong to the $\omega$-limit set of its Teichmüller horocycle flow orbit? The corresponding question for homogeneous spaces has a negative answer due to Ratner.
\end{question}

\begin{question}
	Give a quantitative estimate for the spectral gap of the $\mathrm{SL}(2,\mathbf{R})$-representation of any strata of Abelian or quadratic differentials. In particular, what is the optimal polynomial mixing rate for the Teichmüller horocycle flow?
\end{question}

\subsection*{Outline of the paper.} In \S 2 we discuss the abstract methods for measurable flows adapted from the work of Bourgain and Katz. These methods work in great generality and the proofs assume the least possible hypotheses. In \S 3 we apply these methods to the Teichmüller horocycle flow on $\mathrm{SL}(2,\mathbf{R})$-invariant subvarieties of Abelian or quadratic differentials after introducing and discussing some well known aspects about the dynamics of these flows.

\subsection*{Acknowledgments.} The author is very grateful to Alex Wright and Steve Kerckhoff for their invaluable advice, patience, and encouragement. The author would also like to thank Manfred Einsiedler for introducing him to the work of Asaf Katz. This work got started while the author was participating in the \textit{Dynamics: Topology and Numbers} trimester program at the Hausdorff Research Institute for Mathematics (HIM). The author is very grateful for the hospitality of the HIM and for the hard work of the organizers of the trimester program. This work was finished while the author was a member of the Institute for Advanced Study (IAS). The author is very grateful to the IAS for its hospitality. This material is based upon work supported by the National Science Foundation under Grant No. DMS-1926686.

\section{Polynomially mixing flows}

\subsection*{Outline of this section.} In this section we state and prove the main result that will be used in the proofs of Theorems \ref{theo:main} and \ref{theo:sub_main}. This result provides a non-trivial bound on the Hausdorff dimension of the set of points whose orbit under a polynomially-sub-divergent and polynomially-mixing flow does not equidistribute with respect to a given sub-uniform measure. See Theorem \ref{theo:main_general}.

\subsection*{Statement of main result.} Let $(X,d)$ be a metric space, $\Phi := \{\phi_t \colon X \to X\}_{t \in \mathbf{R}}$ be a flow on $X$, and $\alpha > 0$. We say that $\Phi$ is $(d,\alpha)$-polynomially-sub-divergent if for every compact subset $K \subseteq X$ there exists a constant $C =  C(K) > 0$ such that for every $x,y \in K$ and every $t > 1$ satisfying $\phi_t.x \in K$,
\begin{equation}
\label{eq:non_div}
d(\phi_t.x,\phi_t.y) \leq C \cdot t^\alpha \cdot d(x,y).
\end{equation}
An example of a polynomially-sub-divergent flow is the horocycle flow of any complete, finite area hyperbolic surface.

Let $(X,d)$ be a metric space. Denote by $B(x,r) \subseteq X$ the open ball of radius $r > 0$ centered at $x \in X$. Let $\mu$ be a Borel measure on $X$ and $\beta > 0$. We say that $\mu$ is $(d,\beta)$-sub-uniform if for every compact subset $K \subseteq X$ there exist constants $c = c(K) > 0$ and $r_0 = r_0(K) > 0$ such that for every $x \in K$ and every $0 < r < r_0$,
\begin{equation}
\label{eq:sub_unif}
\mu(B(x,r)) \geq c \cdot r^\beta.
\end{equation}
Any smooth measure on any smooth Riemannian manifold is sub-uniform. Using the Vitali covering lemma one can show that if $X$ supports a $(d,\beta)$-sub-uniform measure then its Hausdorff dimension satisfies
\[
\dimH(X) \leq \beta.
\]

Let $(X,\mathcal{B})$ be a measurable space, $\Phi := \{\phi_t \colon X \to X\}_{t \in \mathbf{R}}$ be a measurable flow on $X$, $\mu$ be a $\Phi$-invariant probability measure on $X$, $f \colon X \to \mathbf{R}$ be a bounded measurable function on $X$ with zero $\mu$-average, and $\gamma > 0$. We say $\Phi$ is $(f,\mu,\gamma)$-polynomially-mixing if there exists a constant $C > 0$ such that for every $t > 1$,
\begin{equation}
\label{eq:mix}
\left|\langle f, f\circ \phi_t\rangle_{L^2(\mu)}\right| := \left|\int_X f(x) \cdot f \circ \phi_t(x) \thinspace d\mu(x) \right| \leq C \cdot t^{-\gamma}.
\end{equation}
More generally, given a bounded, Borel measurable function $f \colon X \to \mathbf{R}$, we say that $\Phi$ is $(f,\mu,\gamma)$-polynomially-mixing if it is $(\of,\mu,\gamma)$-polynomially-mixing for the zero-$\mu$-average normalization $\of \colon X \to \mathbf{R}$ of $f$. By work of Ratner \cite{Ra87}, the horocycle flow on any compact, finite area hyperbolic surface $X$ is polynomially mixing with respect to any sufficiently regular observable $f \colon X \to \mathbf{R}$. The rate of mixing $\gamma > 0$ is uniformly controlled by the spectral gap of the corresponding $\mathrm{SL}(2,\mathbf{R})$-representation, or, in the compact case, the Laplace-Beltrami operator.

Let $(X,d)$ be a metric space. Denote by $\mathcal{C}_0(X)$ the space of continuous, compactly supported functions on $X$ endowed with the sup-norm topology. Let $\Phi := \{\phi_t \colon X \to X\}_{t \in \mathbf{R}}$ be a flow on $X$ and $\mu$ be a Borel probability measure on $X$. We are interested in the Hausdorff dimension of the set of points $E(\Phi,\mu) \subseteq X$ which do not equidistribute with respect to $\mu$, i.e.,
\[
E(\Phi,\mu) := \left\lbrace x \in X \colon \exists f \in \mathcal{C}_0(X), \ \limsup_{T \to \infty} \left|\frac{1}{T}\int_0^T \hspace{-0.2cm} f(\phi_t.x) \thinspace dt - \int_X \hspace{-0.1cm} f(x) \thinspace d\mu(x) \right| > 0 \right\rbrace.
\]

The following theorem is the main result of this section.

\begin{theorem}
	\label{theo:main_general}
	Let $(X,d)$ be a $\sigma$-compact metric space, $\alpha  > 0$, $\Phi := \{\phi_t \colon X \to X\}_{t \in \mathbf{R}}$ be a Borel measurable, $(d,\alpha)$-polynomially-sub-divergent flow, $\beta > 0$, and $\mu$ be a $\Phi$-invariant, $(d,\beta)$-sub-uniform Borel probability measure on $X$. Suppose there exists $\gamma > 0$ and a countable set $S \subseteq \mathcal{C}_0(X)$ of compactly supported Lipschitz functions such that $\Phi$ is $(f,\mu,\gamma)$-polynomially-mixing for every $f \in S$. Then,
	\[
	\dimH(E(\Phi,\mu)) \leq \beta - \alpha^{-1} \cdot \min\{1,\gamma\}.
	\]
\end{theorem}

\subsection*{Sketch of proof.} Let $(X,d)$ be a metric space, $\Phi := \{\phi_t \colon X \to X\}_{t \in \mathbf{R}}$ be a Borel measurable flow on $X$, $\mu$ be a Borel probability measure on $X$, and $f \colon X \to \mathbf{R}$ be a bounded, Borel measurable function on $X$. Denote by $E(\Phi,\mu,f) \subseteq X$ the set
\[
E(\Phi,\mu,f) := \left\lbrace x \in X \colon \limsup_{T \to \infty} \left|\frac{1}{T}\int_0^T f(\phi_t.x) \thinspace dt - \int_X f(x) \thinspace d\mu(x) \right| > 0 \right\rbrace.
\]

Theorem \ref{theo:main_general} will be deduced from the following preliminary result by means of a standard approximation argument.


\begin{theorem}
	\label{theo:main_func}
	Let $(X,d)$ be a $\sigma$-compact metric space, $\alpha > 0$, $\Phi := \{\phi_t \colon X \to X\}_{t \in \mathbf{R}}$ be a Borel measurable, $(d,\alpha)$-polynomially-sub-divergent flow, $\beta > 0$, $\mu$ be a $\Phi$-invariant, $(d,\beta)$-sub-uniform Borel probability measure on $X$, $\gamma > 0$, and $f \colon X \to \mathbf{R}$ be a Lipschitz function constant outside of a compact set such that $\Phi$ is $(f,\mu,\gamma)$-polynomially-mixing. Then, 
	\[
	\dimH(E(\Phi,\mu,f)) \leq \beta - \alpha^{-1} \cdot \min\{1,\gamma\}.
	\]
\end{theorem}

We now give an outline of the proof Theorem \ref{theo:main_func}. For simplicity we consider the case when $X$ is compact. Assume without loss of generality that $f \colon X \to \mathbf{R}$ has zero $\mu$-average. For every $T > 0$ let $A_T f \colon X \to \mathbf{R}$ be the orbit averaging function which to every $x \in X$ assigns the value
\begin{equation}
\label{eq:average}
(A_T f)(x) := \frac{1}{T} \int_0^T f(\phi_t.x) \thinspace dt.
\end{equation}
For every $\epsilon > 0$ consider the set $E(\Phi,\mu,f,\epsilon) \subseteq X$ given by
\[
E(\Phi,\mu,f,\epsilon) := \left\lbrace x \in X \colon \limsup_{m \to \infty} |(A_{(1+\epsilon)^m} f)(x)| > 0 \right\rbrace.
\]
An argument in the spirit of ideas introduced by Bourgain \cite{B89} shows that the equidistribution of orbits of a flow can be studied by considering times diverging along slowly lacunary geometric sequences. More concretely,
\[
E(\Phi,\mu,f) = \bigcup_{n \in \mathbf{N}} E\left(\Phi,\mu,f,n^{-1}\right).
\]
As the Hausdorff dimension of a countable union of sets is equal to the supremum of the Hausdorff dimensions of the sets, 
\[
\dimH\left(E(\Phi,\mu,f)\right) = \sup_{n \in \mathbf{N}} \dimH \left(E\left(\Phi,\mu,f,n^{-1}\right)\right).
\]
Thus, it is enough for our purposes to bound the Hausdorff dimension of the sets $E(\Phi,\mu,f,\epsilon) \subseteq X$ for every $\epsilon > 0$.

Bounding the Hausdorff dimension of a set can be achieved by ensuring the sumability of certain series. To guarantee we have quantitative control over the terms in these series we incorporate the rates of equidistribution into our estimates. More concretely, for every $\kappa > 0$ consider the set $E(\Phi,\mu,f,\epsilon,\kappa) \subseteq X$ given by
\[
E(\Phi,\mu,f,\epsilon,\kappa) := \left\lbrace x \in X \colon \limsup_{m \to \infty} \frac{|(A_{(1+\epsilon)^m} f)(x)|}{(1+\epsilon)^{-\kappa m}}  > 1 \right\rbrace.
\]
Directly from the definitions one can check that
\[
E(\Phi,\mu,f,\epsilon) \subseteq \bigcap_{\kappa > 0} E\left(\Phi,\mu,f,\epsilon,\kappa\right).
\]
By the monotonicity of Hausdorff dimension it follows that
\[
\dimH\left(E(\Phi,\mu,f,\epsilon)\right) \leq \inf_{\kappa > 0} \dimH\left(E\left(\Phi,\mu,f,\epsilon,\kappa\right)\right).
\]
Thus, it is enough for our purposes to bound the Hausforff dimension of the set $E(\Phi,\mu,f,\epsilon,\kappa) \subseteq X$ for $\kappa > 0$ arbitrarily small.

Consider for every $m \in \mathbf{N}$ the set $E(\Phi,\mu,f,\epsilon,\kappa,m) \subseteq X$ given by
\begin{equation}
\label{eq:E}
E(\Phi,\mu,f,\epsilon,\kappa,m) := \left\lbrace x \in X \colon |(A_{(1+\epsilon)^m} f)(x)| > (1+\epsilon)^{-\kappa m} \right\rbrace.
\end{equation}
Directly from the definitions one can check that
\[
E\left(\Phi,\mu,f,\epsilon,\kappa\right) = \bigcap_{M \in \mathbf{N}} \bigcup_{m \geq M} E(\Phi,\mu,f,\epsilon,\kappa,m).
\]
Thus, to bound the Hausdorff dimension of the set $E\left(\Phi,\mu,f,\epsilon,\kappa\right) \subseteq X$, it is enough to construct tight covers of the sets $E(\Phi,\mu,f,\epsilon,\kappa,m) \subseteq X$ by balls of radii converging to zero as $m \to \infty$.

Given $\delta > 0$ let $F \subseteq E(\Phi,\mu,f,\epsilon,\kappa,m)$ be a maximal $\delta$-separated subset, i.e., $d(x,y) \geq \delta$ for every $x,y \in F$. The maximality of $F$ guarantees that
\[
E(\Phi,\mu,f,\epsilon,\kappa,m) \subseteq \bigcup_{x \in F} B(x,\delta).
\]
The tightness of this cover is measured by the cardinality of $F$. The polynomial-sub-divergence of $\Phi$ and the fact that $f \colon X \to \mathbf{R}$ is Lipschitz ensure that points sufficiently close to $E(\Phi,\mu,f,\epsilon,\kappa,m)$ belong to $E(\Phi,\mu,f,\epsilon,\kappa',m)$ for some slightly larger $\kappa' > \kappa$. It follows that, for $\delta > 0$ sufficiently small, 
\[
\bigsqcup_{x \in F} B(x,\delta/2) \subseteq E(\Phi,\mu,f,\epsilon,\kappa',m).
\]
This union is disjoint because $F$ is $\delta$-separated. Thus, as $\mu$ is sub-uniform, to bound the cardinality of $F$, it is enough to bound the $\mu$-measure of  $E(\Phi,\mu,f,\epsilon,\kappa',m) \subseteq X$.

Chebyshev's inequality ensures that
\[
\mu(E(\Phi,\mu,f,\epsilon,\kappa,m)) \leq (1+\epsilon)^{2\kappa m} \cdot \| A_{(1+\epsilon)^m} f \|_{L^2(\mu)}^2.
\]
The $L^2(\mu)$-norm of the orbit averaging function $A_Tf \colon X \to \mathbf{R}$ can be controlled for every $T > 1$ using the polynomial mixing rate of $\Phi$ with respect to $f$ and $\mu$. A careful analysis of the critical exponent that ensures the sumability of these estimates yields the desired bound on the Hausdorff dimension of $E(\Phi,\mu,f) \subseteq X$.

\subsection*{Hausdorff dimension.} We review some basic aspects of the theory of Hausdorff dimension. Let $(X,d)$ be a metric space, $S \subseteq X$ be an arbitrary subset, and $\beta \geq 0$. The $\beta$-dimensional Hausdorff measure of $S$ is given by
\[
\mathcal{H}^\beta(S) := \lim_{r \to 0} \inf_\mathcal{U} \sum_{i \in \mathbf{N}} \mathrm{diam}(U_i)^\beta,
\]
where the infimum runs over all open covers $\mathcal{U} := (U_i)_{i \in \mathbf{N}}$ of $S$ admitting the bound $\mathrm{diam}(\mathcal{U}) := \sup_{i \in \mathbf{N}}(U_i) < r$. The Hausdorff dimension of $S$ is given by
\[
\dimH(S) := \inf \left\lbrace\beta \geq 0 \colon \mathcal{H}^\beta(S) = 0 \right\rbrace.
\]
The Hausdorff dimension satisfies the following basic property: Given an arbitrary countable collection $(S_i)_{i \in \mathbf{N}}$ of subsets of $X$,
\begin{equation}
\label{eq:dim_cup}
\dimH\left( \bigcup_{i \in \mathbf{N}} S_i \right) = \sup_{i \in \mathbf{N}} \dimH(S_i).
\end{equation}
It also satisfies the following monotonicity property: If $A \subseteq B \subseteq X$, then
\begin{equation}
\label{eq:monotone}
\dimH(A) \leq \dimH(B).
\end{equation}

\subsection*{Bourgain's argument.} Let $(X,\mathcal{B})$ a measurable space, $\Phi := \{\phi_t \colon X \to X\}_{t \in \mathbf{R}}$ a measurable flow on $X$, and $f \colon X \to \mathbf{R}$ a bounded, measurable function on $X$. As in (\ref{eq:average}), for every $T > 0$ consider the orbit average function $A_T f \colon X \to \mathbf{R}$ which to every $x \in X$ assigns the value
\[
(A_T f)(x) := \frac{1}{T}\int_0^T f(\phi_t.x) \thinspace dt.
\]

The following result inspired by ideas of Bourgain \cite{B89} shows that the equidistribution of orbits of a flow can be studied by considering times diverging along slowly lacunary geometric sequences. 

\begin{proposition}
	\label{prop:bourg}
	Let $(X,\mathcal{B},\mu)$ be a measure space and $f \colon X \to \mathbf{R}$ be a bounded measurable function with zero $\mu$-average. Then, for every $x \in X$, the function $T > 0 \mapsto (A_T f)(x)$ converges to $0$ as $T \to \infty$ if and only if for every $\epsilon > 0$ the sequence $((A_{(1+\epsilon)^m} f)(x))_{m \in \mathbf{N}}$ converges to $0$ as $m \to \infty$.
\end{proposition}

\begin{proof}
	Fix $x \in X$. Let us prove the non-trivial implication. Given $T > 1$ denote by $m = m(T) \in \mathbf{N}$ the unique non-negative integer such that $(1+\epsilon)^m \leq T < (1+\epsilon)^{m+1}$. This condition guarantees
	\begin{equation}
	\label{eq:c1}
	T - (1+\epsilon)^m < (1+\epsilon)^{m+1} - (1+\epsilon)^m \leq \epsilon \cdot T.
	\end{equation}
	A direct application of the triangle inequality shows that
	\begin{gather}
	\left| (A_T f)(x)  - (A_{(1+\epsilon)^m} f)(x)\right| 	\label{eq:c2} \\
	\leq \left| \frac{1}{T} \int_{(1+\epsilon)^m}^{T} f(u_t.x) \thinspace dt \right| + \left|  \left( \frac{1}{T} - \frac{1}{(1+\epsilon)^m}\right) \int_{0}^{(1+\epsilon)^m} \hspace{-0.15cm} f(u_t.x) \thinspace dt \right|. \nonumber
	\end{gather}
	Using (\ref{eq:c1}) the first term can be bounded directly in terms of the sup-norm $\|f\|_\infty$,
	\begin{equation}
	\label{eq:c3}
	\left| \frac{1}{T} \int_{(1+\epsilon)^m}^{T} f(u_t.x) \thinspace dt \right| \leq \|f\|_\infty  \cdot \left(\frac{T- (1+\epsilon)^m}{T}\right) \leq \|f\|_\infty \cdot \epsilon.
	\end{equation}
	Using (\ref{eq:c1}) the second term can be bounded as follows,
	\begin{equation}
	\label{eq:c4}
	\left|  \left( \frac{1}{T} - \frac{1}{(1+\epsilon)^m}\right) \int_{0}^{(1+\epsilon)^m} \hspace{-0.15cm}f(u_t.x) \thinspace dt \right| \leq  \|f\|_\infty  \cdot \left(\frac{T- (1+\epsilon)^m}{T}\right) \leq \|f\|_\infty \cdot \epsilon.
	\end{equation}
	Putting together (\ref{eq:c2}), (\ref{eq:c3}), and (\ref{eq:c4}) we deduce
	\[
	\left| (A_T f)(x)  - (A_{(1+\epsilon)^m} f)(x)\right| \leq 2 \cdot \|f \|_\infty \cdot \epsilon.
	\]
	In particular, by the triangle inequality,
	\begin{align}
	\left| (A_T f)(x) \right| &\leq \left| (A_T f)(x)  - (A_{(1+\epsilon)^m} f)(x)\right| +  \left|(A_{(1+\epsilon)^m} f)(x)\right| \label{eq:c5}\\
	&\leq 2 \cdot \|f\|_\infty \cdot \epsilon + \left|(A_{(1+\epsilon)^m} f)(x)\right|. \nonumber
	\end{align}
	
	Now let $\delta > 0$ be arbitrary. Choose $\epsilon > 0$ small enough so that $2 \cdot \|f\|_\infty \cdot \epsilon \leq \delta /2$ and $T_0  > 1$ large enough so that $m_0 := m(T_0)$ satisfies $\left|(A_{(1+\epsilon)^m} f)(x)\right| \leq \delta/2$ for every $m \geq m_0$. The bound (\ref{eq:c5}) then guarantees $\left| (A_T f)(x) \right| \leq \delta$ for every $T > T_0$. As $\delta > 0$ is arbitrary, this finishes the proof.
\end{proof}

\subsection*{Clustering.} Let $(X,d)$ be a metric space, $\Phi := \{\phi_t \colon X \to X\}_{t \in \mathbf{R}}$ be a Borel measurable flow, $\mu$ be a $\Phi$-invariant Borel probability measure on $X$, $f \colon X \to \mathbf{R}$ be a bounded, Borel measurable function with zero $\mu$-average, $T > 0$, and $\kappa > 0$. Consider the set $G(\Phi,\mu,f,T,\kappa) \subseteq X$ given by
\[
G(\Phi,\mu,f,T,\kappa) := \left\lbrace x \in X \colon |(A_T f)(x)| \leq T^{-\kappa} \right\rbrace.
\]

\begin{proposition}
	\label{prop:clustering}
	Let $(X,d)$ be a metric space, $K \subseteq X$ be a compact subset, $\alpha > 0$, $\Phi := \{\phi_t \colon X \to X\}_{t \in \mathbf{R}}$ be a Borel measurable, $(d,\alpha)$-polynomially-sub-divergent flow, $\mu$ be a $\Phi$-invariant Borel probability measure on $X$, and $f \colon X \to \mathbf{R}$ be a Lipschitz function with zero $\mu$-average and constant outside of a compact set. Then, there exists a constant $D > 1$ with the following property: if $x,y \in K$, $\kappa \in (0,1)$, and $T > 1$ satisfy $x \in G(\Phi,\mu,f,T,\kappa)$ and $d(x,y) \leq T^{-\alpha-\kappa}$, then $y \in G(\Phi,\mu,f,T,\kappa')$ for $\kappa' := \kappa - \log_T(D)$.
\end{proposition}

\begin{proof}
	Let $x,y \in K$, $\kappa \in (0,1)$, and $T > 1$ be such that $x \in G(\Phi,\mu,f,T,\kappa)$ and $d(x,y) \leq T^{-\alpha-\kappa}$. Using the triangle inequality we can write
	\begin{align}
	\label{eq:d1}
	|(A_Tf)(y) - (A_Tf)(x)| \leq &\frac{1}{T} \int_0^1 |f(\phi_t.y) - f(\phi_t.x)| \thinspace dt \\
	&+ \frac{1}{T} \int_1^T |f(\phi_t.y) - f(\phi_t.x)| \thinspace dt. \nonumber
	\end{align}
	The first term can be bounded directly in terms of the sup-norm $\|f\|_\infty$,
	\begin{equation}
	\label{eq:d2}
	\frac{1}{T} \int_0^1 |f(\phi_t.y) - f(\phi_t.x)| \leq 2 \cdot \|f\|_\infty \cdot T^{-1}.
	\end{equation}
	Denote by $K' \subseteq X$ the compact set outside of which $f \colon X \to \mathbf{R}$ is constant. Let $C = C(K \cup K') > 0$ be as in (\ref{eq:non_div}). Using the fact that $f$ is Lipschitz, the polynomial-sub-divergence of $\Phi$, and the assumption that $d(x,y) \leq T^{-\alpha-\kappa}$, we deduce that, for every $t > 1$, if either $\phi_t.x \in K$ or $\phi_t.y \in K$, then,
	\begin{align*}
	|f(\phi_t.y) - f(\phi_t.x)| &\leq \|f\|_\mathrm{Lip} \cdot d(\phi_t.y,\phi_t.x) \\
	&\leq \|f\|_\mathrm{Lip} \cdot C \cdot t^\alpha \cdot d(x,y) \nonumber\\
	&\leq \|f\|_\mathrm{Lip} \cdot C \cdot t^\alpha \cdot T^{-\alpha-\kappa}. \nonumber
	\end{align*}
	In particular, the second term in (\ref{eq:d1}) can be bounded as follows,
	\begin{equation}
	\label{eq:d3}
	\frac{1}{T} \int_1^T |f(\phi_t.y) - f(\phi_t.x)| \thinspace dt \leq \frac{C \cdot \|f\|_\mathrm{Lip}}{\alpha+1} \cdot T^{-\kappa}.
	\end{equation}
	Putting together (\ref{eq:d1}), (\ref{eq:d2}), and (\ref{eq:d3}) we deduce
	\[
	|(A_Tf)(y) - (A_Tf)(x)| \leq \left(2 \cdot \|f\|_\infty + \frac{C \cdot \|f\|_\mathrm{Lip}}{\alpha+1} \right) \cdot T^{-\kappa}.
	\]
	In particular, by the triangle inequality and the assumption $x \in G(\Phi,\mu,f,T,\kappa)$,
	\begin{align*}
	|(A_Tf)(y)| &\leq |(A_Tf)(y) - (A_Tf)(x)| + |(A_Tf)(x)| \\
	&\leq \left(1 + 2 \cdot \|f\|_\infty + \frac{C \cdot \|f\|_\mathrm{Lip}}{\alpha+1} \right) \cdot T^{-\kappa}. \nonumber
	\end{align*}
	The condition $y \in G(\Phi,\mu,f,T,\kappa')$ with $\kappa' := \kappa - \log_T(D)$ holds by letting
	\[
	D := 1 + 2 \cdot \|f\|_\infty + \frac{C \cdot \|f\|_\mathrm{Lip}}{\alpha+1}. \qedhere
	\]
\end{proof}

Directly from Proposition \ref{prop:clustering} we deduce the following corollary.

\begin{corollary}
	\label{cor:clustering}
	Let $(X,d)$ be a metric space, $\alpha > 0$, $K \subseteq X$ be a compact subset, $\Phi := \{\phi_t \colon X \to X\}_{t \in \mathbf{R}}$ be a Borel measurable, $(d,\alpha)$-polynomially-sub-divergent flow, $\mu$ be a $\Phi$-invariant Borel probability measure on $X$, and $f \colon X \to \mathbf{R}$ be a Lipschitz function with zero $\mu$-average and constant outside of a compact set. Then, there exists a constant $D > 1$ with the following property: if $x,y \in K$, $\kappa \in (0,1)$, and $T > 1$ satisfy $x \in X \backslash G(\Phi,\mu,f,T,\kappa)$ and $d(x,y) \leq D^{-1} \cdot T^{-\alpha-\kappa}$, then $y \in X \backslash G(\Phi,\mu,f,T,\kappa')$ for $\kappa' := \kappa + \log_T(D)$.
\end{corollary}

\subsection*{Orbit average variance.} The following result shows that the variance of any orbit averaging function can be controlled in terms of the polynomial mixing rate of the corresponding flow.

\begin{proposition}
	\cite[Lemma 3.1]{R86}
	\label{prop:var}
	Let $(X,\mathcal{B},\mu)$ be a probability space, $\Phi := \{\phi_t \colon X \to X\}_{t \in \mathbf{R}}$ be a measurable flow on $X$, $\gamma \in (0,1)$, and $f \colon X \to \mathbf{R}$ be a bounded, measurable function with zero $\mu$-average such that $\Phi$ is $(f,\mu,\gamma)$-polynomially-mixing. Then, there exists a constant $C > 0$ such that for every $T > 1$,
	\[
	\| A_T f \|_{L^2(\mu)}^2 \leq C \cdot T^{-\gamma}.
	\]
\end{proposition}

\begin{proof}
	Fix $T > 1$. A direct application of Fubini's theorem shows that
	\begin{equation*}
	\| A_T f \|_{L^2(\mu)}^2 = \frac{1}{T^2} \int_0^T \int_0^T \int_X f(\phi_t.x) \thinspace f(\phi_s.x) \thinspace d\mu(x) \thinspace ds \thinspace dt.
	\end{equation*}
	Using the $\Phi$-invariance of $\mu$ we can rewrite this formula as
	\begin{equation}
	\label{eq:a2}
	\| A_T f \|_{L^2(\mu)}^2 = \frac{1}{T^2} \int_0^T \int_0^T \int_X f(x) \thinspace f(\phi_{s-t}.x) \thinspace d\mu(x) \thinspace ds \thinspace dt.
	\end{equation}
	Consider the change of variables $(u,v) \colon \mathbf{R}^2 \to \mathbf{R}^2$ given by
	\[
	u(t,s) := s+t, \quad v(t,s) := s-t.
	\]
	The change of variables formula applied to (\ref{eq:a2}) yields
	\[
	\| A_T f \|_{L^2(\mu)}^2 = \frac{1}{2T^2} \int_0^{2T} \int_{|v|\leq \min\{u,2T-u\}} \int_X f(x) \thinspace f(\phi_v.x) \thinspace d\mu(x) \thinspace dv \thinspace du.
	\]
	Let us rewrite this equality as
	\[
	\| A_T f \|_{L^2(\mu)}^2 = \frac{1}{2T^2} \int_0^{2T} \int_{|v|\leq \min\{u,2T-u\}} \langle f, f\circ \phi_v \rangle_{L^2(\mu)} \thinspace dv \thinspace du.
	\]
	We decompose this integral into two terms,
	\begin{align}
	\label{eq:a3}
	\| A_T f \|_{L^2(\mu)}^2 = &\frac{1}{2T^2} \int_0^{2T} \int_{|v|\leq \min\{u,2T-u,1\}} \langle f, f\circ \phi_v \rangle_{L^2(\mu)} \thinspace dv \thinspace du\\
	&+\frac{1}{2T^2} \int_1^{2T-1} \int_{1 \leq |v|\leq \min\{u,2T-u\}} \langle f, f\circ \phi_v \rangle_{L^2(\mu)} \thinspace dv \thinspace du. \nonumber
	\end{align}
	The first term can be bounded directly in terms of the sup-norm $\|f\|_\infty$,
	\begin{equation}
	\label{eq:a4}
	\frac{1}{2T^2} \int_0^{2T} \int_{|v|\leq \min\{u,2T-u,1\}} \langle f, f\circ \phi_v \rangle_{L^2(\mu)} \thinspace dv \thinspace du \leq 2 \cdot \|f\|^2_\infty \cdot T^{-1}.
	\end{equation}
	The second term can be bounded using the fact that $\Phi$ is $(f,\mu,\gamma)$-polynomially-mixing. Indeed, let $C > 0$ be as in (\ref{eq:mix}). It follows that
	\begin{gather}
	\label{eq:a5}
	\frac{1}{2T^2} \int_1^{2T-1} \int_{1 \leq |v|\leq \min\{u,2T-u\}} \langle f, f\circ \phi_v \rangle_{L^2(\mu)} \thinspace dv \thinspace du \\
	\leq \frac{C}{2T^2} \int_1^{2T-1} \int_{1 \leq |v|\leq \min\{u,2T-u\}} |v|^{-\gamma}  \thinspace dv \thinspace du. \nonumber
	\end{gather}
	A direct computation shows that
	\begin{equation}
	\label{eq:a6}
	\frac{C}{2T^2} \int_1^{2T-1} \int_{1 \leq |v|\leq \min\{u,2T-u\}} |v|^{-\gamma}  \thinspace dv \thinspace du \leq \frac{2 C}{(-\gamma+1)(-\gamma+2)} \cdot T^{-\gamma}.
	\end{equation}
	From (\ref{eq:a5}) and (\ref{eq:a6}) we deduce
	\begin{gather}
	\frac{1}{2T^2} \int_1^{2T-1} \int_{1 \leq |v|\leq \min\{u,2T-u\}} \langle f, f\circ \phi_v \rangle_{L^2(\mu)} \thinspace dv \thinspace du \label{eq:a7}\\ \leq \frac{2 C}{(-\gamma+1)(-\gamma+2)} \cdot T^{-\gamma}. \nonumber
	\end{gather}
	Putting together (\ref{eq:a3}), (\ref{eq:a4}), and (\ref{eq:a7}) we conclude 
	\[
	\| A_T f \|_{L^2(\mu)}^2  \leq \left( 2 \cdot \|f\|_\infty^2 +  \frac{2 C}{(-\gamma+1)(-\gamma+2)} \right)  \cdot T^{-\gamma}. \qedhere
	\]
	\end{proof}

\subsection*{Proofs of the main results.} We are now ready to prove Theorem \ref{theo:main_func}, which we restate here for the reader's convenience.  Recall that for $(X,d)$ a metric space, $\Phi := \{\phi_t \colon X \to X\}_{t \in \mathbf{R}}$ a Borel measurable flow on $X$, $\mu$ a Borel probability measure on $X$, and $f \colon X \to \mathbf{R}$ a bounded, Borel measurable function on $X$, we consider
\[
E(\Phi,\mu,f) := \left\lbrace x \in X \colon \limsup_{T \to \infty} \left|\frac{1}{T}\int_0^T f(\phi_t.x) \thinspace dt - \int_X f(x) \thinspace d\mu(x) \right| > 0 \right\rbrace.
\]

\begin{theorem}
	\label{theo:main_func_2}
	Let $(X,d)$ be a $\sigma$-compact metric space, $\alpha > 0$, $\Phi := \{\phi_t \colon X \to X\}_{t \in \mathbf{R}}$ be a Borel measurable, $(d,\alpha)$-polynomially-sub-divergent flow, $\beta > 0$, $\mu$ be a $\Phi$-invariant, $(d,\beta)$-sub-uniform Borel probability measure on $X$, $\gamma > 0$, and $f \colon X \to \mathbf{R}$ be a Lipschitz function constant outside of a compact set such that $\Phi$ is $(f,\mu,\gamma)$-polynomially-mixing. Then, 
	\begin{equation}
	\label{eq:goal}
	\dim_H(E(\Phi,\mu,f)) \leq \beta - \alpha^{-1} \cdot \min\{1,\gamma\}.
	\end{equation}
\end{theorem}

\begin{proof}
	Assume without loss of generality that $f \colon X \to \mathbf{R}$ has zero $\mu$-average. For simplicity assume also that $\gamma \in (0,1)$. Recall that for every $\epsilon > 0$ we consider 
	\[
	E(\Phi,\mu,f,\epsilon) := \left\lbrace x \in X \colon \limsup_{m \to \infty} |(A_{(1+\epsilon)^m} f)(x)| > 0 \right\rbrace.
	\]
	By Proposition \ref{prop:bourg}, we can write $E(\Phi,\mu,f) \subseteq X$ as the following countable union,
	\[
	E(\Phi,\mu,f) = \bigcup_{n \in \mathbf{N}} E\left(\Phi,\mu,f,n^{-1}\right).
	\]
	Using property (\ref{eq:dim_cup}) of Hausdorff dimension we deduce
	\[
	\dimH\left( E(\Phi,\mu,f) \right) = \sup_{n \in \mathbf{N}} \dimH\left(E\left(\Phi,\mu,f,n^{-1}\right) \right).
	\]
	Thus, to prove (\ref{eq:goal}), it is enough to show that for every $\epsilon > 0$,
	\begin{equation}
	\label{eq:p1}
	\dim_H(E(\Phi,\mu,f,\epsilon)) \leq \beta - \alpha^{-1} \cdot \gamma.
	\end{equation}
	As $X$ is $\sigma$-compact, the same argument guarantees that, to prove (\ref{eq:p1}), it is enough to show that for every compact subset $K \subseteq X$,
	\begin{equation}
	\label{eq:p_1}
	\dim_H(E(\Phi,\mu,f,\epsilon) \cap K) \leq \beta - \alpha^{-1} \cdot \gamma.
	\end{equation}
	
	Fix $K \subseteq X$ compact and $\epsilon > 0$. Recall that for every $\kappa > 0$  we consider the set
	\[
	E(\Phi,\mu,f,\epsilon,\kappa) := \left\lbrace x \in X \colon \limsup_{m \to \infty} \frac{|(A_{(1+\epsilon)^m} f)(x)|}{(1+\epsilon)^{-\kappa m}}  > 1 \right\rbrace.
	\]
	Directly from the definitions one can check that
	\[
	E(\Phi,\mu,f,\epsilon) \subseteq \bigcap_{\kappa > 0} E\left(\Phi,\mu,f,\epsilon,\kappa\right).
	\]
	Using property (\ref{eq:monotone}) of Hausdorff dimension we deduce
	\begin{equation}
	\label{eq:inf}
	\dim_H(E(\Phi,\mu,f,\epsilon) \cap K) \leq \inf_{\kappa > 0} \dim_H(E(\Phi,\mu,f,\epsilon,\kappa) \cap K) 
	\end{equation}
	Thus, to prove (\ref{eq:p_1}), it is enough to bound the Hausdorff dimension of the set $E(\Phi,\mu,f,\epsilon,\kappa) \cap K$ for $\kappa > 0$ arbitrarily small. 
	
	Fix $\kappa > 0$. Recall that for every $m \in \mathbf{N}$ we consider the set
	\[
	E(\Phi,\mu,f,\epsilon,\kappa,m) := \left\lbrace x \in X \colon |(A_{(1+\epsilon)^m} f)(x)| > (1+\epsilon)^{-\kappa m} \right\rbrace.
	\]
	Directly from the definitions one can check that
	\begin{equation}
	\label{eq:wcov}
	E\left(\Phi,\mu,f,\epsilon,\kappa\right) = \bigcap_{M \in \mathbf{N}} \bigcup_{m \geq M} E(\Phi,\mu,f,\epsilon,\kappa,m).
	\end{equation}
	Thus, to bound the Hausdorff dimension of $E\left(\Phi,\mu,f,\epsilon,\kappa\right) \cap K$, it is enough to construct tight covers of the sets $E(\Phi,\mu,f,\epsilon,\kappa,m) \cap K$ by balls of radii converging to zero as $m \to \infty$.
	
	Let $m \in \mathbf{N}$ be arbitrary and $T = T(m) := (1+\epsilon)^m > 1$. Denote by $D > 0$ the constant provided by Corollary \ref{cor:clustering} for the compact subset $K \subseteq X$. Consider $\delta = \delta(T) := D^{-1} \cdot T^{-\alpha-\kappa} > 0$. Let $c = c(K) > 0$ and $r_0 = r_0(K) > 0$ be as in (\ref{eq:sub_unif}). Assume $m \in \mathbf{N}$ is large enough so that $\delta = \delta(T(m)) < r_0$. Let $F \subseteq E(\Phi,\mu,f,\epsilon,\kappa,m)\cap K$ be a maximal $\delta$-separated subset. Then,
	\[
	E(\Phi,\mu,f,\epsilon,\kappa,m)\cap K \subseteq \bigcup_{x \in F} B(x,\delta).
	\]
	Let $\kappa' = \kappa'(T) := \kappa + \log_T(D) > 0$. By Corollary \ref{cor:clustering},
	\[
	\bigsqcup_{x \in F} B(x,\delta/2) \subseteq E(\Phi,\mu,f,\epsilon,\kappa',m).
	\]
	This union is disjoint because $F$ is $\delta$-separated. Using the countable additivity and of $\mu$ and the fact that $\mu$ is $(d,\beta)$-sub-uniform we deduce
	\begin{equation}
	\label{eq:p2}
	\# F \leq 2^\beta \cdot c^{-1} \cdot \delta^{-\beta} \cdot \mu(E(\Phi,\mu,f,\epsilon,\kappa',m)).
	\end{equation}
	Chebyshev's inequality ensures that
	\begin{equation}
	\label{eq:p3}
	\mu(E(\Phi,\mu,f,\epsilon,\kappa',m)) \leq (1+\epsilon)^{2\kappa' m} \cdot \| A_{(1+\epsilon)^m} f \|_{L^2(\mu)}^2.
	\end{equation}
	Let $C > 0$ be as in Proposition \ref{prop:var}. It follows that
	\begin{equation}
	\label{eq:p4}
	\| A_{(1+\epsilon)^m} f \|_{L^2(\mu)}^2 \leq C \cdot (1+\epsilon)^{-\gamma m}.
	\end{equation}
	Putting together (\ref{eq:p2}), (\ref{eq:p3}), and (\ref{eq:p4}) we deduce
	\begin{equation}
	\label{eq:card}
	\#F \leq 2^\beta \cdot C \cdot D^{\beta} \cdot c^{-1} \cdot (1+\epsilon)^{(\alpha\beta + \kappa \beta + 2\kappa' - \gamma) m}.
	\end{equation}
	
	Let $\xi > 0$ be arbitrary. Consider $M \in \mathbf{N}$ such that $\delta_m := \delta(T(m)) < r_0$ and $\kappa_m' := \kappa + \log_{T(m)}(D) < \kappa + \xi$ for every $m \geq M$. For every $m \geq M$ let $F_m \subseteq E(\Phi,\mu,f,\epsilon,\kappa,m) \cap K$ be a maximal $\delta_m$-separated subset. Consider the open cover $\mathcal{U}_M$ of $E\left(\Phi,\mu,f,\epsilon,\kappa\right) \cap K$ given by
	\[
	E\left(\Phi,\mu,f,\epsilon,\kappa\right) \cap K \subseteq \bigcup_{m \geq M} \bigcup_{x \in F_m} B(x,\delta_m).
	\]
	The collection $\mathcal{U}_M$ is a cover of $E\left(\Phi,\mu,f,\epsilon,\kappa\right) \cap K$ because of (\ref{eq:wcov}). This cover has diameter $\mathrm{diam}(\mathcal{U}_M) = \delta_M \to 0$ as $M \to \infty$. It follows that, for every $\eta> 0$,
	\[
	\mathcal{H}^\eta(E\left(\Phi,\mu,f,\epsilon,\kappa\right)\cap K) \leq \lim_{M \to \infty} \sum_{m \geq M} \#F_m \cdot \delta_m^\eta.
	\]
	From this bound and (\ref{eq:card}) we deduce that, for every $\eta > 0$,
	\[
	\mathcal{H}^\eta(E\left(\Phi,\mu,f,\epsilon,\kappa\right)) \leq 2^\beta \cdot C \cdot D^{\beta-\eta} \cdot c^{-1} \cdot \lim_{M \to \infty} \sum_{m \geq M}  (1+\epsilon)^{(\alpha\beta + \kappa \beta + 2\kappa' -\gamma - \alpha \eta - \kappa \eta) m}.
	\]
	It follows that, for every $\eta > 0$,  $\mathcal{H}^\eta(E\left(\Phi,\mu,f,\epsilon,\kappa\right)) = 0$ whenever 
	\[
	\sum_{m \in \mathbf{N}}  (1+\epsilon)^{(\alpha \beta + \kappa \beta + 2\kappa + 2 \xi - \gamma - \alpha\eta -\kappa\eta)m} < \infty.
	\]
	The critical value of $\eta$ for which this condition holds is given by
	\[
	\eta = \frac{\alpha\beta + \kappa\beta + 2\kappa+2\xi-\gamma}{\alpha + \kappa}.
	\]
	It follows that
	\[
	\dimH (E\left(\Phi,\mu,f,\epsilon,\kappa\right)\cap K) \leq \frac{\alpha\beta + \kappa\beta + 2\kappa+2\xi-\gamma}{\alpha + \kappa}.
	\]
	As $\xi > 0$ is arbitrary we deduce
	\[
	\dimH (E\left(\Phi,\mu,f,\epsilon,\kappa\right)\cap K) \leq \frac{\alpha\beta + \kappa\beta + 2\kappa-\gamma}{\alpha + \kappa}.
	\]
	It follows from this and (\ref{eq:inf}) that
	\[
	\dimH (E\left(\Phi,\mu,f,\epsilon\right)\cap K) \leq \beta - \alpha^{-1} \cdot \gamma.
	\]
	This finishes the proof of (\ref{eq:p_1}) and thus of (\ref{eq:goal}).
\end{proof}

We are now ready to prove Theorem \ref{theo:main_general}, the main result of this section, which we restate here for the reader's convenience. Recall that for $(X,d)$ a metric space, $\Phi := \{\phi_t \colon X \to X\}_{t \in \mathbf{R}}$ a Borel measurable flow on $X$, and $\mu$ a Borel probability measure on $X$, we consider the set
\[
E(\Phi,\mu) := \left\lbrace x \in X \colon \exists f \in \mathcal{C}_0(X), \ \limsup_{T \to \infty} \left|\frac{1}{T}\int_0^T \hspace{-0.2cm} f(\phi_t.x) \thinspace dt - \int_X \hspace{-0.1cm} f(x) \thinspace d\mu(x) \right| > 0 \right\rbrace.
\]

\begin{theorem}
	Let $(X,d)$ be a $\sigma$-compact metric space, $\alpha  > 0$, $\Phi := \{\phi_t \colon X \to X\}_{t \in \mathbf{R}}$ be a Borel measurable, $(d,\alpha)$-polynomially-sub-divergent flow, $\beta > 0$, and $\mu$ be a $\Phi$-invariant, $(d,\beta)$-sub-uniform Borel probability measure on $X$. Suppose there exists $\gamma > 0$ and a countable set $S \subseteq \mathcal{C}_0(X)$ of compactly supported Lipschitz functions such that $\Phi$ is $(f,\mu,\gamma)$-polynomially-mixing for every $f \in S$. Then,
	\[
	\dimH(E(\Phi,\mu)) \leq \beta - \alpha^{-1} \cdot \min\{1,\gamma\}.
	\]
\end{theorem}

\begin{proof}
	We begin by showing that
	\begin{equation}
	\label{eq:q0}
	E(\Phi,\mu) = \bigcup_{f \in S} E(\Phi,\mu,f).
	\end{equation}
	Let us prove the non-trivial inclusion. Suppose $x \in E(\Phi,\mu)$. Then, there exists a continuous, compactly supported function $f \colon X \to \mathbf{R}$ such that
	\begin{equation}
	\label{eq:q1}
	\limsup_{T \to \infty} \left|\frac{1}{T}\int_0^T f(\phi_t.x) \thinspace dt - \int_X f(x) \thinspace d\mu(x) \right| := \epsilon > 0.
	\end{equation}
	As $S \subseteq \mathcal{C}_0(X)$ is dense, there exists a compactly supported Lipschitz function $g \colon X \to \mathbf{R}$ in $S$ such that $\|f - g \|_\infty \leq \epsilon/3$. In particular,
	\[
	\left|\int_X f(x) \thinspace d\mu(x) - \int_X g(x) \thinspace d\mu(x) \right| \leq \epsilon/3.
	\]
	Analogously, for every $T > 0$,
	\[
	\left| \frac{1}{T} \int_0^T f(\phi_t.x) \thinspace dt - \frac{1}{T} \int_0^T g(\phi_t.x) \thinspace dt \right| \leq \epsilon/3.
	\]
	It follows from the triangle inequality that, for every $T > 0$,
	\begin{align*}
	\left|\frac{1}{T}\int_0^T g(\phi_t.x) \thinspace dt - \int_X g(x) \thinspace d\mu(x) \right| \geq \left|\frac{1}{T}\int_0^T f(\phi_t.x) \thinspace dt - \int_X f(x) \thinspace d\mu(x) \right| - 2\epsilon/3.
	\end{align*}
	Taking $\limsup$ as $T \to \infty$ and using (\ref{eq:q1}) we deduce
	\[
	\limsup_{T \to \infty} \left|\frac{1}{T}\int_0^T g(\phi_t.x) \thinspace dt - \int_X g(x) \thinspace d\mu(x) \right| \geq \epsilon/3 > 0.
	\]
	It follows that $x \in E(\Phi,\mu,g)$ with $g \in S$, thus proving the desired inclusion.
	
	As $S \subseteq \mathcal{C}_0(X)$ is countable, (\ref{eq:q0}) and property (\ref{eq:dim_cup}) guarantee 
	\[
	\dimH\left(E(\Phi,f)\right) = \sup_{f \in S} \dimH\left(E(\Phi,\mu,f) \right).
	\]
	Using Theorem \ref{theo:main_func} we conclude
	\[
	\dimH\left(E(\Phi,f)\right) \leq \beta - \alpha^{-1} \cdot \min\{1,\gamma\}. \qedhere
	\]
\end{proof}

\subsection*{A remark on time changes.} Given $(X,d)$ a metric space, $\Phi := \{\phi_t \colon X \to X\}_{t \in \mathbf{R}}$ a Borel measurable flow on $X$, $\mu$ a Borel probability measure on $X$, $f \colon X \to \mathbf{R}$ a bounded, Borel measurable function on $X$, and $\rho > 0$, consider the set 
\[
E^\rho(\Phi,\mu,f) := \left\lbrace x \in X \colon \limsup_{T \to \infty} \left|\frac{1}{T}\int_0^T f(\phi_{t^\rho}.x) \thinspace dt - \int_X f(x) \thinspace d\mu(x) \right| > 0 \right\rbrace.
\]

The methods used in the proof of Theorem \ref{theo:main_func} also yield the following result.

\begin{theorem}
	\label{theo:time_change_1}
	Let $(X,d)$ be a $\sigma$-compact metric space, $\alpha > 0$, $\Phi := \{\phi_t \colon X \to X\}_{t \in \mathbf{R}}$ be a Borel measurable, $(d,\alpha)$-polynomially-sub-divergent flow, $\beta > 0$, $\mu$ be a $\Phi$-invariant, $(d,\beta)$-sub-uniform Borel probability measure on $X$, $\gamma > 0$, $f \colon X \to \mathbf{R}$ be a Lipschitz function constant outside of a compact set such that $\Phi$ is $(f,\mu,\gamma)$-polynomially-mixing, and $\rho > 0$. Then, 
	\begin{equation}
	\dim_H(E^\rho(\Phi,\mu,f)) \leq \beta - \rho^{-1} \cdot \alpha^{-1} \cdot \min\{1,\rho \cdot \gamma\}.
	\end{equation}
\end{theorem}

Given $(X,d)$ a metric space, $\Phi := \{\phi_t \colon X \to X\}_{t \in \mathbf{R}}$ a Borel measurable flow on $X$, and $\mu$ a Borel probability measure on $X$, consider the set
\[
E^\rho(\Phi,\mu) := \left\lbrace x \in X \colon \exists f \in \mathcal{C}_0(X), \ \limsup_{T \to \infty} \left|\frac{1}{T}\int_0^T \hspace{-0.2cm}f(\phi_{t^\rho}.x) dt - \int_X \hspace{-0.1cm} f(x) d\mu(x) \right| > 0 \right\rbrace.
\]

The following result can be deduced from Theorem \ref{theo:time_change_1} in the same way Theorem \ref{theo:main_general} can be deduced from Theorem \ref{theo:main_func}.

\begin{theorem}
	Let $(X,d)$ a $\sigma$-compact metric space, $\alpha  > 0$, $\Phi := \{\phi_t \colon X \to X\}_{t \in \mathbf{R}}$ a Borel measurable, $(d,\alpha)$-polynomially-sub-divergent flow, $\beta > 0$, $\mu$ a $\Phi$-invariant, $(d,\beta)$-sub-uniform Borel probability measure on $X$, and $\rho > 0$. Suppose there exists $\gamma > 0$ and a countable set $S \subseteq \mathcal{C}_0(X)$ of compactly supported Lipschitz functions such that $\Phi$ is $(f,\mu,\gamma)$-polynomially-mixing for every $f \in S$. Then,
	\[
	\dimH(E^\rho(\Phi,\mu)) \leq \beta - \rho^{-1} \cdot \alpha^{-1} \cdot \min\{1,\rho \cdot \gamma\}.
	\]
\end{theorem}

\section{The Teichmüller horocycle flow}

\subsection*{Outline of this section.} In this section we apply Theorem \ref{theo:main_general} to the Teichmüller horocycle flow on $\mathrm{SL}(2,\mathbf{R})$-invariant subvarieties of Abelian or quadratic differentials to prove Theorems \ref{theo:main} and \ref{theo:sub_main}, the main results of this paper. 

\subsection*{The AGY metric.} To simplify the notation we specialize the following discussion to $\mathrm{SL}(2,\mathbf{R})$-invariant subvarieties of Abelian differentials. The same constructions and results apply to $\mathrm{SL}(2,\mathbf{R})$-invariant subvarieties of quadratic differentials by considering holonomy double covers.

Let $\mathcal{M}$ be an $\mathrm{SL}(2,\mathbf{R})$-invariant subvariety of Abelian differentials. Denote points in $\mathcal{M}$ by $(X,\omega)$, where $X$ is a Riemann surface and $\omega$ is an Abelian differential on $X$. The tangent space of $\mathcal{M}$ at $(X,\omega) \in \mathcal{M}$ identifies with a real vector subspace $V$ of the relative cohomology group $H^1(X,\Sigma;\mathbf{C})$, where $\Sigma \subseteq X$ denotes the set of zeroes of $\omega$. For $v \in V$ consider the norm
\[
\|v\|_\mathrm{AGY} := \sup_{\gamma \in \Gamma} \left| \frac{v(\gamma)}{\mathrm{hol}_\omega(\gamma)} \right|,
\]
where $\Gamma$ denotes the set of saddle connections of $\omega$ and $\mathrm{hol}_\omega(\gamma) \in \mathbf{C}$ denotes the holonomy of the saddle connection $\gamma$ with respect to $\omega$. By work of Avila, Gouëzel, and Yoccoz \cite{AGY06}, this definition indeed gives rise to a norm on $V$ and the corresponding Finsler metric on $\mathcal{M}$ is complete. We refer to this metric as the AGY metric of $\mathcal{M}$ and denote it by $d_\mathrm{AGY}$.

\begin{proposition}\cite[Corollary 2.6]{CSW20}
	\label{prop:1}
	Let $\mathcal{M}$ be an $\mathrm{SL}(2,\mathbf{R})$-invariant subvariety of Abelian or quadratic differentials. Then, the Teichmüller horocycle flow on $\mathcal{M}$ is $(d_\mathrm{AGY},2)$-polynomially-sub-divergent.
\end{proposition}

\begin{proof}
	To simplify the notation we assume $\mathcal{M}$ is an $\mathrm{SL}(2,\mathbf{R})$-invariant subvariety of Abelian differentials. Denote by $U := \{u_t \colon \mathcal{M} \to \mathcal{M} \}_{t \in \mathbf{R}}$ the Teichmüller horocycle flow on $\mathcal{M}$. Let $(X,\omega) \in \mathcal{M}$ and $\Sigma \subseteq X$ be the set of zeroes of $\omega$. Given $t \in \mathbf{R}$, denote by $u_t.\omega$ the Abelian differential of $u_t.(X,\omega) \in \mathcal{M}$. Given $\gamma$ a saddle connection of $\omega$ and $t \in \mathbf{R}$, denote by $u_t.\gamma$ the parallel transport of $\gamma$ to $u_t.(X,\omega) \in \mathcal{M}$. Given $v \in H^1(X,\Sigma;\mathbf{C})$ a tangent vector of $\mathcal{M}$ at $(X,\omega)$ and $t \in \mathbf{R}$, denote by $du_t.v$ the derivative of $u_t$ applied to $v$. Notice that, for every saddle connection $\gamma$ of $\omega$, every tangent vector $v \in H^1(X,\Sigma;\mathbf{C})$ of $\mathcal{M}$ at $(X,\omega)$, and every $t \in \mathbf{R}$,
	\begin{gather*}
	(du_t.v)(u_t.\gamma) = \Re(v(\gamma)) + t \Im(v(\gamma)) + i \Im(v(\gamma)),\\
	\mathrm{hol}_{u_t.\omega}(u_t.\gamma) = \Re(\mathrm{hol}_\omega(\gamma)) + t \Im(\mathrm{hol}_\omega(\gamma)) + i \Im(\mathrm{hol}_\omega(\gamma)).
	\end{gather*}
	A direct computation shows that, for every $x,y \in \mathbf{R}$ and every $t \in \mathbf{R}$,
	\[
	\frac{1}{\sqrt{2} \cdot (1+|t|)} \cdot |x + iy|	\leq |x+ty+iy| \leq \sqrt{2} \cdot (1+|t|) \cdot |x + iy|.
	\]
	Using these inequalities we deduce that, for every tangent vector $v \in H^1(X,\Sigma;\mathbf{C})$ of $\mathcal{M}$ at $(X,\omega)$ and every $t > 1$, if $\Gamma$ denotes the set of saddle connections of $\omega$,
	\begin{align*}
	\|du_t.v\|_\mathrm{AGY} &= \sup_{\gamma \in \Gamma} \left| \frac{(du_t.v)(u_t.\gamma)}{\mathrm{hol}_{u_t.\omega}(u_t.\gamma)} \right| \\
	&\leq 2 \cdot (1+|t|)^2 \cdot \sup_{\gamma \in \Gamma} \left| \frac{v(\gamma)}{\mathrm{hol}_\omega(\gamma)} \right| \\
	&\leq 8 \cdot t^2 \cdot \|v\|_{\mathrm{AGY}}.
	\end{align*}
	Integrating this infinitesimal inequality over piecewise smooth paths we conclude
	\[
	d_\mathrm{AGY}(u_t.x, u_t.y) \leq 8\cdot t^2 \cdot d_\mathrm{AGY}(x,y). \qedhere
	\]
\end{proof}

\subsection*{Affine measures.} By work of Eskin and Mirzakhani \cite{EMir18}, and Eskin, Mirzakhani, and Mohammadi \cite{EMM15}, $\mathrm{SL}(2,\mathbf{R})$-orbit-closures on strata of Abelian differentials are cut out in period coordinates by homogeneous polynomials with real coefficients; by work of Filip \cite{F16}, these orbit closures are real analytic subvarieties. Every $\mathrm{SL}(2,\mathbf{R})$-invariant subvariety $\mathcal{M}$ of Abelian or quadratic differentials can be endowed with a natural $\mathrm{SL}(2,\mathbf{R})$-invariant, smooth, Lebesgue class probability measure $\mu$. We refer to this measure as the affine measure of $\mathcal{M}$. The following result follows directly from the fact that affine measures are smooth and Lebesgue class, and the fact that the AGY metric is Finsler.

\begin{proposition}
	\label{prop:2}
	Let $\mathcal{M}$ be a $\beta$-dimensional $\mathrm{SL}(2,\mathbf{R})$-invariant subvariety of Abelian or quadratic differentials and $\mu$ be the affine measure of $\mathcal{M}$. Then, $\mu$ is $(d_\mathrm{AGY},\beta)$-sub-uniform.
\end{proposition}

\subsection*{Spectral gap.} Building on previous work of Avila, Gouëzel, and Yoccoz \cite{AGY06}, and Avila and Resende \cite{AR12}, Avila and Gouëzel \cite{AG13} showed that every $\mathrm{SL}(2,\mathbf{R})$-invariant subvariety of Abelian or quadratic differentials has positive spectral gap with respect to its affine measure. By work of Ratner \cite{Ra87}, this implies that the Teichmüller horocycle flow on any $\mathrm{SL}(2,\mathbf{R})$-invariant subvariety of Abelian or quadratic differentials is polynomially mixing with respect to sufficiently regular observables. More concretely, the following holds.

\begin{theorem}
	\label{theo:1}
	Let $\mathcal{M}$ be an $\mathrm{SL}(2,\mathbf{R})$-invariant subvariety of Abelian or quadratic differentials, $U := \{u_t \colon \mathcal{M} \to \mathcal{M}\}_{t \in \mathbf{R}}$ be the Teichmüller horocycle flow on $\mathcal{M}$, $\mu$ be the affine measure of $\mathcal{M}$, and $f \colon \mathcal{M} \to \mathbf{R}$ be a smooth, compactly supported function on $\mathcal{M}$. Then, $U$ is $(f,\mu,\gamma)$-polynomially-mixing for some constant $\gamma \in (0,1)$ depending only on the spectral gap of $\mathcal{M}$.
\end{theorem}

	We refer to the constant $\gamma  \in (0,1)$ in Theorem \ref{theo:1} as the polynomial mixing rate of the Teichmüller horocycle flow with respect to the affine measure of $\mathcal{M}$.

\subsection*{Proof of the main result.} Recall that for $(X,d)$ a metric space, $\Phi := \{\phi_t \colon X \to X\}_{t \in \mathbf{R}}$ a Borel measurable flow on $X$, and $\mu$ a Borel probability measure on $X$, 
\[
E(\Phi,\mu) := \left\lbrace x \in X \colon \exists f \in \mathcal{C}_0(X), \ \limsup_{T \to \infty} \left|\frac{1}{T}\int_0^T \hspace{-0.2cm} f(\phi_t.x) \thinspace dt - \int_X \hspace{-0.1cm} f(x) \thinspace d\mu(x) \right| > 0 \right\rbrace.
\]

Given any $\mathrm{SL}(2,\mathbf{R})$-invariant subvariety of Abelian or quadratic differentials $\mathcal{M}$, one can always find a dense set $S \subseteq \mathcal{C}_0(\mathcal{M})$ of smooth, compactly supported, and, in particular, Lipschitz functions on $\mathcal{M}$. Thus, Propositions \ref{prop:1} and \ref{prop:2}, and Theorem \ref{theo:1}, allow us to apply Theorem \ref{theo:main_general} to the Teichmüller horocycle flow on any $\mathrm{SL}(2,\mathbf{R})$-invariant subvariety of Abelian or quadratic differentials to deduce the following more precise version of Theorem \ref{theo:sub_main}.

\begin{theorem}
	\label{theo:2}
	Let $\mathcal{M}$ be a $\beta$-dimensional, $\mathrm{SL}(2,\mathbf{R})$-invariant subvariety of Abelian or quadratic differentials, $U := \{u_t \colon \mathcal{M} \to \mathcal{M}\}_{t \in \mathbf{R}}$ be the Teichmüller horocycle flow on $\mathcal{M}$, $\mu$ be the affine measure of $\mathcal{M}$, and $\gamma \in (0,1)$ be the polynomial mixing rate of $U$ with respect to $\mu$. Then,
	\[
	\dimH(E(U,\mu)) \leq \beta - \gamma/2.
	\]
\end{theorem}

Theorem \ref{theo:main} now follows directly from Theorem \ref{theo:2} and the fact that affine measures on $\mathrm{SL}(2,\mathbf{R})$-invariant subvarieties of Abelian or quadratic differentials have full support. The following is a more precise statement.

\begin{theorem}
	\label{theo:main_2}
	Let $\mathcal{M}$ be a $\beta$-dimensional, $\mathrm{SL}(2,\mathbf{R})$-invariant subvariety of Abelian or quadratic differentials, $U := \{u_t \colon \mathcal{M} \to \mathcal{M}\}_{t \in \mathbf{R}}$ be the Teichmüller horocycle flow on $\mathcal{M}$, $\mu$ be the affine measure of $\mathcal{M}$, and $\gamma \in (0,1)$ be the polynomial mixing rate of $U$ with respect to $\mu$. Then, for every $x \in \mathcal{M}$ such that $\smash{\overline{U.x}} \neq X$,
	\[
	\dimH\left(\overline{U.x}\right) \leq \beta - \gamma/2.
	\]
\end{theorem}

\subsection*{Further remarks.} In \cite{CSW20}, Chaika, Smillie, and Weiss constructed examples of Abelian differentials whose Teichmüller horocycle flow orbit equidistributes with respect to a measure that does not contain them in their support. Furthermore, they constructed examples of Abelian differentials whose Teichmüller horocycle flow orbit does not equidistribute with respect to any measure.  As a direct consequence of Theorem \ref{theo:2}, we see that, on any $\mathrm{SL}(2,\mathbf{R})$-invariant subvariety of Abelian or quadratic differentials, the set of points exhibiting any of these pathological behaviors has Hausdorff dimension bounded as in Theorem \ref{theo:main_2}. The same applies to the set of points which do not belong to the $\omega$-limit set of their Teichmüller horocycle flow orbit; see Question \ref{question:2}.


\bibliographystyle{amsalpha}


\bibliography{bibliography}

\end{document}